\documentclass[11pt]{article}
       \usepackage{amsfonts}
       \usepackage{stmaryrd}
       \usepackage{latexsym,amssymb,mathrsfs,fancyhdr}
       \font\tenmsb=msbm10
       \font\sevenmsb=msbm7
       \font\fivemsb=msbm5
       \catcode`\@=11
       \ifx\amstexloaded@\relax\catcode`\@=\active
       \endinput\else\let\amstexloaded@\relax\fi
       \def\spaces@{\space\space\space\space\space}
       \def\spaces@@{\spaces@\spaces@\spaces@\spaces@\spaces@}
       \def\space@.  {\futurelet\space@\relax}
       \space@.   %
       \def\Err@#1{\errhelp\defaulthelp@\errmessage{AmS-TeX error: #1}}
       \def\relaxnext@{\let\next\relax}
       \def\accentfam@{7}
       \def\noaccents@{\def\accentfam@{0}}
       \def\Cal{\relaxnext@\ifmmode\let\next\Cal@\else
       \def\next{\Err@{Use \string\Cal\space only in math mode}}\fi\next}
       \def\Cal@#1{{\Cal@@{#1}}}
       \def\Cal@@#1{\noaccents@\fam\tw@#1}
       \def\Bbb{\relaxnext@\ifmmode\let\next\Bbb@\else
       \def\next{\Err@{Use \string\Bbb\space only in math mode}}\fi\next}
       \def\Bbb@#1{{\Bbb@@{#1}}}
       \def\Bbb@@#1{\noaccents@\fam\msbfam#1}
       \newfam\msbfam
       \textfont\msbfam=\tenmsb
       \scriptfont\msbfam=\sevenmsb
       \scriptscriptfont\msbfam=\fivemsb

\usepackage{dsfont}
\usepackage[german,english]{babel}
\usepackage{amsmath,amssymb}
\usepackage[square, comma, sort&compress, numbers]{natbib}
\usepackage{cases}
\usepackage{mathrsfs}
\usepackage{amsmath}
\usepackage{amsthm}
\usepackage{amsfonts}
\usepackage{amssymb}
\usepackage{latexsym}
\usepackage{fancyhdr}
\usepackage{extarrows}
\usepackage{geometry}
\newtheorem{thm}{Theorem}[section]
\newtheorem{prop}[thm]{Proposition}
\newtheorem{lem}[thm]{Lemma}
\newtheorem{rem}[thm]{Remark}
\newtheorem{iteration lemma}[thm]{iteration Lemma}

\newtheorem{defn}[thm]{Definition}

\newtheorem*{acknowledgements*}{ACKNOWLEDGEMENTS}

\begin{document}

\setlength{\columnsep}{5pt}
\title{\bf Pseudo Core Inverses in Rings with Involution}
\author{Yuefeng Gao\footnote{ E-mail: yfgao91@163.com},
 \ Jianlong Chen\footnote{ Corresponding author. E-mail: jlchen@seu.edu.cn } \\
School of Mathematics, Southeast University, \\ Nanjing 210096,  China}
     \date{}

\maketitle
\begin{quote}
{\textbf{Abstract: }\small Let $R$ be a ring with involution. In this paper, we introduce a new type of generalized inverse called pseudo core inverse in $R$. The notion of core inverse was introduced by Baksalary and Trenkler for matrices of index 1 in 2010 and then it was generalized to an arbitrary $*$-ring case by Raki\'{c}, Din\v{c}i\'{c} and Djordjevi\'{c} in 2014. Our definition of pseudo core inverse extends the notion of core inverse to elements of an arbitrary index in $R$.  Meanwhile, it generalizes the notion of core-EP inverse, introduced by Manjunatha Prasad and Mohana for matrices in 2014, to the case of $*$-ring. Some equivalent characterizations for elements in $R$ to be pseudo core invertible are given and expressions are presented especially in terms of Drazin inverse and \{1,3\}-inverse.
Then, we investigate the relationship between pseudo core inverse and  other generalized inverses. Further, we establish several properties of the pseudo core inverse. Finally, the computations for pseudo core inverses of matrices are exhibited.


}
\end{quote}

\section{ Introduction }\label{a}
Throughout this paper, $R$ always denotes a ring with involution; we say $*$-ring for short.
An involution $*$ in $R$ is an anti-isomorphism satisfying $$(a^{*})^{*}=a,~~(a+b)^{*}=a^{*}+b^{*},~~(ab)^*=b^*a^*~\text{for~all}~a, b \in R.$$


The Moore-Penrose inverse of $a\in R$, denoted by $a^{\dag}$, is the unique solution to the following Penrose equations
$$(1)~axa=a,~~(2)~xax=x,~~(3)~(ax)^{*}=ax,~~(4)~(xa)^{*}=xa.$$
$x$ satisfying equations $(1)$ and~$(3)$ is called a $\{1,3\}$-inverse of $a$, denoted by $a^{(1,3)}$.

The symbol $a^D\in R$ stands for the Drazin inverse of $a
\in R$, i.e., the unique element satisfying the following equations
$$a^ka^Da=a^k~\text{for some positive integer}~k,~~a^Daa^D=a^D,~~aa^D=a^Da.$$
The smallest positive integer $k$ satisfying above equations is called the Drazin index of $a$, denoted by $i(a)$. If $i(a)=1$, then the Drazin inverse of $a$ is called the group inverse of $a$ and is denoted by $a^{\#}$. One can see \cite{D1974} for a deep study of generalized inverses.

Baksalary and Trenkler \cite{D2010} introduced the notion of core inverse for complex matrices in 2010.
Later, Raki\'{c}, Din\v{c}i\'{c} and Djordjevi\'{c} \cite{DDD2014} generalized this notion to an arbitrary $*$-ring case. They proved that for $a\in R$, the core inverse of $a$ is the unique element $a^{\tiny\textcircled{\tiny \#}}$ satisfying the following five equations
$$aa^{\tiny\textcircled{\tiny \#}}a=a,~~a^{\tiny\textcircled{\tiny \#}}aa^{\tiny\textcircled{\tiny \#}}=a^{\tiny\textcircled{\tiny \#}},~~(aa^{\tiny\textcircled{\tiny \#}})^{*}=aa^{\tiny\textcircled{\tiny \#}},~~a^{\tiny\textcircled{\tiny \#}}a^2=a,~~ a(a^{\tiny\textcircled{\tiny \#}})^2=a^{\tiny\textcircled{\tiny \#}}.$$

Let us denote by $R^{\dag},~R^{\{1,3\}},~R^{D},~R^{\#}$~and~$R^{\tiny\textcircled{\tiny \#}}$~the set of all Moore-Penrose invertible, \{1,3\}-invertible, Drazin invertible, group invertible and core invertible elements in $R$, respectively.

Recently, Xu, Chen and Zhang \cite{DD2016} characterized the core invertible elements in $R$ in terms of three equations. The core inverse of $a$ is the unique solution to equations $$xa^2=a,~~ax^2=x~~\text{and}~~(ax)^{*}=ax.$$
Further, they pointed out that $a\in R^{\tiny\textcircled{\tiny \#}}$ if and only if $a\in R^{\#}$ and $a\in R^{\{1,3\}}$, in which case, $a^{\tiny\textcircled{\tiny \#}}=a^{\#}aa^{(1,3)}$.

In 2014, Baksalary and Trenkler \cite{D2014}, Manjunatha Prasad and Mohana \cite{DDDD2014} introduced respectively the concepts of generalized core inverse and core-EP inverse which exist for arbitrary square complex matrices, based on the notion of core inverse restricted to complex matrices of index 1(i.e., rank($A)=$rank($A^2$)). Let $A, G\in \mathbb{C}^{n\times n}$, where the symbol $\mathbb{C}^{n\times n}$ stands for the set of all $n\times n$ complex matrices,
$G$ is a core-EP inverse of $A$ if $GAG=G$ and
$$\mathcal{C}(G)=\mathcal{C}(G^*)=\mathcal{C}(A^d),$$ where $d$ is the index of $A$ (i.e., the smallest positive integer such that rank$(A^{d})=$rank$(A^{d+1})$), and $\mathcal{C}(G)$ is the column space of $G$.

Motivated by the above two papers, we put forward the notion of pseudo core inverse in $R$ as a generalization for both core inverse in $R$ and core-EP inverse for complex matrices.

\begin{defn}\label{1.1} Let $a \in R$. If there exists~$x \in R$~such that~
$$(\mathrm I)~xa^{m+1}=a^m~\text{for some positive integer}~m,~(\mathrm{II})~ax^2=x,~(\mathrm{III})~(ax)^{*}=ax,$$
then we call $a$ pseudo core invertible.
\end{defn}
By Theorem \ref{1} below, equations $(\mathrm{I})-(\mathrm{III})$ determine $x$ uniquely when it exists, so we may refer to $x$ as the pseudo core inverse of $a$, denoted by $a^{\scriptsize\textcircled{\tiny D}}$. The smallest  positive integer $m$ satisfying equations $(\mathrm{I})-(\mathrm{III})$ is called the pseudo core index of $a$, denoted by $I(a)$.\\

Similarly we define dual pseudo core inverse as follows:
\begin{defn}\label{1.2}~Let $a \in R$. The dual pseudo core inverse of $a$, denoted by $a_{\scriptsize\textcircled{\tiny D}}$, is the unique element $x\in R$ satisfying the following three equations~$$(\mathrm I^{\prime})~a^{m+1}x=a^m~\text{for some positive integer}~m,~({\mathrm {II}}^{\prime})~x^2a=x,~({\mathrm {III}}^{\prime})~(xa)^{*}=xa.$$
The smallest  positive integer $m$ satisfying equations $(\mathrm I^{\prime})$-$({\mathrm {III}}^{\prime})$ is called
the dual pseudo core index of $a$, denoted by $I^{\prime}(a)$.
\end{defn}

Since only multiplication is required in Definition \ref{1.1} and \ref{1.2}, the above two definitions hold, without modification, in $*$-semigroup $S$ (i.e., semigroup $S$ with an involution $*$ satisfying $(a^{*})^{*}=a$ and $(ab)^{*}=b^{*}a^{*}$ for all $a,b\in S$).\\

Here and subsequently,~$R^{\scriptsize\textcircled{\tiny D}}$ and $R_{\scriptsize\textcircled{\tiny D}}$ denote the sets of all pseudo core invertible, dual pseudo core invertible elements in $R$, respectively.

\begin{rem}~$(1)$~If $I(a)=1$~$($resp. $I^{\prime}(a)=1)$, then the pseudo core inverse~$($resp. dual pseudo core inverse$)$ of $a$ is the core inverse~$($resp. dual core inverse$)$ of $a$.\\
$(2)$~$a \in R^{\scriptsize\textcircled{\tiny D}}$ if and only if $a^{*}\in R_{\scriptsize\textcircled{\tiny D}}$. Moreover, $(a^{\scriptsize\textcircled{\tiny D}})^{*}=(a^{*})_{\scriptsize\textcircled{\tiny D}}$.
\end{rem}
In this paper, we mainly consider the pseudo core inverse case. Dual pseudo core inverse case can be investigated analogously.

In Section 2, we compile some basic facts about pseudo core inverses in $R$, such as some equivalent characterizations for an element to be pseudo core invertible are given and expressions are presented.
In Section 3, we reveal the relationship between pseudo core inverse and the inverse along an element as well as that between pseudo core inverse and $(b, c)$-inverse. In Section 4, several properties such as reverse order law and additive property of the pseudo core inverse are obtained.
In the final section, we provide two methods to compute pseudo core inverses for complex matrices.

\vspace{0.5cm}

\section{General results on pseudo core inverses}\label{a}
Several facts about pseudo core inverses in $R$ are established in this section. We begin with an auxiliary lemma.
\begin{lem}\label{2.1}~Let $a \in R$. If there exists~$x \in R$~such that
$$(\mathrm{I})~xa^{m+1}=a^m~\text{for some positive integer}~m,~~(\mathrm{II})~ax^2=x ,$$
then we have the following facts :
\begin{description}
                                    \item[$(1)$] $ax=a^kx^k$ for arbitrary positive integer $k$;
                                    \item[$(2)$] $xax=x$;
                                    \item[$(3)$] $a^kx^ka^k=a^k$ for arbitrary positive integer $k\geq m$;
                                    \item[$(4)$] $a\in R^{D}$, moreover, $a^D=x^{m+1}a^m$ with $i(a)\leq m$.
                                  \end{description}
\end{lem}

\begin{proof}~
(1)~Since~$ax^2=x$, then $ax=a(ax^2)=a^2x^2=a^2(ax^2)x=a^3x^3=\cdots=a^kx^k$ for arbitrary positive integer $k$.
\vspace{2mm}\\
(2)~As~$ax=a^{m+1}x^{m+1}=a^mx^m$ follows from (1), we have $xax=xa^{m+1}x^{m+1}=a^mx^{m+1}=(ax)x=ax^2=x$.
\vspace{2mm}\\
(3)~Let $k\geq m$, then $$a^k=a^ma^{k-m}=xa^{m+1}a^{k-m}=xa^{k+1}=ax^2a^{k+1}=(a^kx^k)xa^{k+1}
=a^k(x^{k+1}a^{k+1}).$$
Thus $a^k=a^kx^ka^k$, on account of $x^{k+1}a^{k+1}=x^ka^k$.
\vspace{2mm}\\
(4)~Let~$b=x^{m+1}a^m$, then we get
\vspace{2mm}\\
\indent$a^mba=a^m(x^{m+1}a^m)a= a^m(x^{m+1}a^{m+1})=a^mx^{m}a^{m}=a^m$;

$bab=(x^{m+1}a^m)a(x^{m+1}a^m)=x^{m+1}a^{m+1}x^{m+1}a^{m}=x^{m}a^{m}x^{m+1}a^{m}=x^{m}(ax^2)a^m\\
\indent\indent\!=x^{m+1}a^m=b$;

$ab=ax^{m+1}a^m=x^ma^m=x^{m+1}a^{m+1}=(x^{m+1}a^m) a=ba$.
\vspace{2mm}\\
\noindent Hence $a^D=b=x^{m+1}a^m$, with $i(a)\leq m$.
\end{proof}
\vspace{4mm}

First and most fundamentally, we have the following theorem.
\begin{thm}\label{1}~
Let~$a\in R$. Then $a$ has at most one pseudo core inverse in $R$.
\end{thm}

\begin{proof}~Suppose~$x~\text{and}~y$~satisfy conditions~(I), (II) and (III)~of our definition of pseudo core inverse, with $m, n$~as pseudo core inverse index respectively.
Let $k=\text{max}\{m, n\}$, by Lemma \ref{2.1}, we have $$x^{k+1}a^k=x^{m+1}a^m=a^D=y^{n+1}a^{n}=y^{k+1}a^{k},$$
and,
$a^kx^ka^k=a^k,~ (a^kx^k)^{*}=(ax)^{*}=ax=a^kx^k$ which yields that $x^k$ is a \{1,3\}-inverse of $a^k$. Likewise,
$y^k$ is one of  the \{1,3\}-inverses of $a^k$. So $a^kx^k=a^k(a^k)^{(1,3)}=a^ky^k$. Then
\begin{equation*}
\begin{aligned}
x&=x(ax)=x(a^kx^k)=x^2a^{k+1}x^{k}=\cdots=x^{k+1}a^{2k}x^{k}=(x^{k+1}a^{k})(a^kx^{k})\\
 &=a^D(a^k(a^k)^{(1,3)})=(y^{k+1}a^{k})(a^ky^{k})=y^{k+1}a^{2k}y^{k}=y.
\end{aligned}
\end{equation*}
Thus $a$ has at most one pseudo core inverse.
\end{proof}
\vspace{4mm}

The following result gives an  equivalent characterization for the existence of the pseudo core inverse as well as the expression in terms of Drazin inverse and \{1,3\}-inverse.
\begin{thm}~Let $a\in R$ and let $k$ be positive integer with $k\geq m$. Then $a \in R^{\scriptsize\textcircled{\tiny D}}$ with $I(a)=m$~if and only if~$a \in R^{D}$ with $i(a)=m$~and~$a^k \in R^{\{1,3\}}$.
In this case, $a^{\scriptsize\textcircled{\tiny D}}=a^Da^k(a^k)^{(1,3)}.$
\end{thm}

\begin{proof}~Suppose $a \in R^{\scriptsize\textcircled{\tiny D}}$ with $I(a)=m$. By Lemma \ref{2.1}, we have~$a\in R^{D}$ with $i(a) \leq m$ and $a^k \in R^{\{1,3\}}$ for arbitrary $k\geq m$. $i(a)<m$ would mean that $xa^m=a^{m-1}$, which is contrary to our definition of $m$.
So $i(a)=m$.

Conversely,~suppose $a \in R^{D}$ with $i(a)=m$~and suppose~$a^k \in R^{\{1,3\}}$ for some $k\geq m$.~
Setting $x=a^Da^k(a^k)^{(1,3)}$, next we prove $a^{\scriptsize\textcircled{\tiny D}}=x$.
In fact,
 \begin{equation*}
 \begin{aligned}
&xa^{m+1}=a^Da^k(a^k)^{(1,3)}a^{m+1}=a^Da^{k}(a^k)^{(1,3)}a^{k+1}(a^D)^{k-m}=a^k(a^D)^{k-m}=a^m;\\ &ax^2=aa^Da^k(a^k)^{(1,3)}a^Da^k(a^k)^{(1,3)}=a^k(a^k)^{(1,3)}a^Da^k(a^k)^{(1,3)}= a^Da^k(a^k)^{(1,3)}=x;\\
&ax=aa^Da^k(a^k)^{(1,3)}=a^k(a^k)^{(1,3)},~\text{so}~(ax)^{*}=ax.
\end{aligned}
\end{equation*}
Hence $a \in R^{\scriptsize\textcircled{\tiny D}}$ with $I(a)\leq m$. $I(a)<m$ is contrary to the definition of $m$.
So~$I(a)=m$.~
\end{proof}

\begin{rem} Theorem 2.3 manifests that $I(a)=i(a)$, whenever $a\in R^{\scriptsize\textcircled{\tiny D}}$.
So, if $a\in R^{\scriptsize\textcircled{\tiny D}}\cap R_{\scriptsize\textcircled{\tiny D}}$, then $I(a)=I^{\prime}(a)=i(a)$.
\end{rem}
\vspace{4mm}
It is well known that $a\in R^D$ if and only if $a^k\in R^{\#}$ for some positive integer $k$ \cite{D2002}; $a\in R^D$~if~and~only~if $a^k\in R^D$ for arbitrary positive integer $k$ if and only if $a^k\in R^D$ for some positive integer $k$ \cite{D1958}; if $a\in R^{D}$, then $a^D\in R^{D}$ with $(a^D)^D=a^2a^D$ \cite{D1958}.

Similarly, we have the following results.
\begin{thm}~Let $a\in R$. Then $a\in R^{\scriptsize\textcircled{\tiny D}}$ if and only if $a^m\in R^{\tiny\textcircled{\tiny \#}}$ for some positive integer $m$. In this case, $(a^m)^{\tiny\textcircled{\tiny \#}}=(a^{\scriptsize\textcircled{\tiny D}})^m$ and $a^{\scriptsize\textcircled{\tiny D}}=a^{m-1}(a^m)^{\tiny\textcircled{\tiny \#}}$.
\end{thm}

\begin{proof}~Suppose $a^{\scriptsize\textcircled{\tiny D}}=x$ with $I(a)=m$. Setting $y=x^m$, by the definition of pseudo core inverse and by Lemma 2.1,
 we can check that
\begin{equation*}
\begin{aligned}
&y(a^m)^2=x^m(a^m)^2=(x^ma^m)a^m=(x^{m+1}a^{m+1})a^m=(x^{m+1}a^m)a^{m+1}=a^Da^{m+1}=a^m;\\
&a^my^2=a^m(x^m)^2=(a^mx^m)x^m=axx^m=ax^{m+1}=x^m=y;\\
&(a^my)^{*}=(a^mx^m)^{*}=(ax)^{*}=ax=a^mx^m=a^my.
\end{aligned}
\end{equation*}
Therefore $(a^m)^{\tiny\textcircled{\tiny \#}}=y=(a^{\scriptsize\textcircled{\tiny D}})^m$.

On the contrary,~since $a^m\in R^{\tiny\textcircled{\tiny \#}}$, by the notion of core inverse, we have
$$(a^m)^{\tiny\textcircled{\tiny \#}}(a^m)^2=a^m,~
~a^m((a^m)^{\tiny\textcircled{\tiny \#}})^2=(a^m)^{\tiny\textcircled{\tiny \#}}~~\text{and}~
~(a^m(a^m)^{\tiny\textcircled{\tiny \#}})^{*}=a^m(a^m)^{\tiny\textcircled{\tiny \#}}.$$
 Let $x=a^{m-1}(a^m)^{\tiny\textcircled{\tiny \#}}$, then we can notice

$xa^{m+1}=a^{m-1}(a^m)^{\tiny\textcircled{\tiny \#}}a^{m+1}=a^{m-1}((a^m)^{\#}a^m)a=a^m;$

$ax^2=a(a^{m-1}(a^m)^{\tiny\textcircled{\tiny \#}})^2
=a^m(a^m)^{\tiny\textcircled{\tiny \#}}a^{m-1}(a^m)^{\tiny\textcircled{\tiny \#}}
=a^m(a^m)^{\tiny\textcircled{\tiny \#}}a^{m-1}a^m((a^m)^{\tiny\textcircled{\tiny \#}})^2$\\
\indent\indent~$=a^{m-1}(a^m)^{\tiny\textcircled{\tiny \#}}=x;$

$(ax)^{*}=(aa^{m-1}(a^m)^{\tiny\textcircled{\tiny \#}})^{*}
=(a^m(a^m)^{\tiny\textcircled{\tiny \#}})^{*}=a^m(a^m)^{\tiny\textcircled{\tiny \#}}
=a(a^{m-1}(a^m)^{\tiny\textcircled{\tiny \#}})=ax.$\\
Hence~$a^{\scriptsize\textcircled{\tiny D}}=x=a^{m-1}(a^m)^{\tiny\textcircled{\tiny \#}}$.
\end{proof}
\vspace{4mm}

\begin{thm}~Let $a \in R$ and let $k$ be positive integer. Then
$a\in R^{\scriptsize\textcircled{\tiny D}}~\text{if~and~only~if}
~a^k\in R^{\scriptsize\textcircled{\tiny D}}.$ In this case, $(a^k)^{\scriptsize\textcircled{\tiny D}}=(a^{\scriptsize\textcircled{\tiny D}})^k$ and $a^{\scriptsize\textcircled{\tiny D}}=a^{k-1}(a^k)^{\scriptsize\textcircled{\tiny D}}$.
\end{thm}

\begin{proof}~~Suppose $a^{\scriptsize\textcircled{\tiny D}}=x$ with $I(a)=m$, then we have
$$xa^{m+1}=a^m,~ax^2=x,~(ax)^{*}=ax.$$
For arbitrary positive integer $k$, let $n$ be the unique integer satisfying $0\leq kn-m<k$, then,
\begin{equation*}
\begin{aligned}
&(a^k)^n=a^{kn}=a^ma^{kn-m}=xa^{m+1}a^{kn-m}=xa^{kn+1},~\mbox{by~induction},~
(a^k)^n=x^k(a^k)^{n+1};\\
&a^k(x^k)^2=(a^kx^k)x^k=(ax)x^k=ax^{k+1}=x^k;\\
&a^kx^k=ax,~ \text{so}~(a^kx^k)^{*}=a^kx^k.
\end{aligned}
\end{equation*}
Thus $(a^k)^{\scriptsize\textcircled{\tiny D}}=x^k=(a^{\scriptsize\textcircled{\tiny D}})^k$, with $I(a^k)\leq n$. $I(a^k)<n$ clearly forces that $x^k(a^k)^n=(a^k)^{n-1}$, and, since $x^ka^{kn}=(a^k)^{\scriptsize\textcircled{\tiny D}}a^{kn}=(a^k)^Da^{kn}=(a^D)^ka^{kn}=a^Da^{kn-k+1}$, which~implies $a^Da^{kn-k+1}=a^{kn-k}$, whence, by $i(a)=I(a)=m$, we should have $kn-k\geq m$, contrary to our definition of $n$. Hence $I(a^k)=n$.
\vspace{2mm}

Conversely, suppose $(a^k)^{\scriptsize\textcircled{\tiny D}}=y$ with $I(a^k)=n$, then we have
$$y(a^k)^{n+1}=(a^k)^n,~~~a^ky^2=y,~~~(a^ky)^{*}=a^ky.$$
Set $x=a^{k-1}y$.~In what follows, we prove $a^{\scriptsize\textcircled{\tiny D}}=x$.
\vspace{2mm}

\mbox{Since}~$xa^{kn+1}=a^{k-1}ya^{kn+1}=a^{k-1}(a^ky^2)a^{kn+1}$, then $xa^{kn+1}= a^{k-1}(a^{kn}y^{n+1})a^{kn+1}$ by induction. So,~we~get~
\begin{equation*}
\begin{aligned}
&xa^{kn+1}=a^{kn+k-1}(y^{n+1}a^{kn})a=a^{kn+k-1}(a^k)^Da=(a^k)^Da^{kn+k}=y^{n+1}(a^k)^{n+1}a^{kn}=y^{n}a^{kn}a^{kn}\\
&~~~~~~~~~=y^{n}(a^{k})^{2n}=a^{kn};\\
&ax^2=a(a^{k-1}y)^2=a^kya^{k-1}y=a^kya^{k-1}(a^ky^2)=a^kya^{k-1}((a^k)^{n+1}y^{n+2}),~ \\
&~~~~~=a^k(y(a^k)^{n+1})a^{k-1}y^{n+2}=a^ka^{kn}a^{k-1}y^{n+2}=a^{k-1}a^{kn}y^{n+1}=a^{k-1}y=x;~~~~~~~~~~~\\
&ax=aa^{k-1}y=a^ky,~ \text{so}~(ax)^{*}=ax.
\end{aligned}
\end{equation*}
From the above, we get $a^{\scriptsize\textcircled{\tiny D}}=x=a^{k-1}(a^k)^{\scriptsize\textcircled{\tiny D}}$, with $I(a)\leq kn$.
\end{proof}
\vspace{4mm}

\begin{thm}~Let $a\in R$. If $a\in R^{\scriptsize\textcircled{\tiny D}}$, then $a^{\scriptsize\textcircled{\tiny D}}\in R^{\scriptsize\textcircled{\tiny D}}$. In fact $a^{\scriptsize\textcircled{\tiny D}}$ is core invertible whenever it exists, and $(a^{\scriptsize\textcircled{\tiny D}})^{\scriptsize\textcircled{\tiny D}}=(a^{\scriptsize\textcircled{\tiny D}})^{\tiny\textcircled{\tiny \#}}=a^2a^{\scriptsize\textcircled{\tiny D}}$.
\end{thm}
\begin{proof}~To prove this, one has merely to verify that if $x$ satisfies $(\mathrm{I}), (\mathrm{II})~\text{and}~(\mathrm{III})$ which defines the pseudo core inverse, then $y=a^2x$ satisfies
$$yx^2=x,~~~~xy^2=y,~~~~(xy)^{*}=xy.$$
Here we omit the details.
\end{proof}
\vspace{4mm}

\begin{prop}~Let $a\in R^{\scriptsize\textcircled{\tiny D}}$. Then $((a^{\scriptsize\textcircled{\tiny D}})^{\scriptsize\textcircled{\tiny D}})^{\scriptsize\textcircled{\tiny D}}=a^{\scriptsize\textcircled{\tiny D}}.$
\end{prop}
\begin{proof}~Suppose $a\in R^{\scriptsize\textcircled{\tiny D}}$ with $I(a)=m$. By Theorem 2.7, we have
 $$((a^{\scriptsize\textcircled{\tiny D}})^{\scriptsize\textcircled{\tiny D}})
 ^{\scriptsize\textcircled{\tiny D}}
=(a^{\scriptsize\textcircled{\tiny D}})^2(a^{\scriptsize\textcircled{\tiny D}})^{\scriptsize\textcircled{\tiny D}}
=(a^{\scriptsize\textcircled{\tiny D}})^2a^2a^{\scriptsize\textcircled{\tiny D}}.$$
Since $(a^{\scriptsize\textcircled{\tiny D}})^2a^2a^{\scriptsize\textcircled{\tiny D}}
=(a^{\scriptsize\textcircled{\tiny D}})^2a^{m+1}(a^{\scriptsize\textcircled{\tiny D}})^m
=a^{\scriptsize\textcircled{\tiny D}}a^m(a^{\scriptsize\textcircled{\tiny D}})^m
=a^{\scriptsize\textcircled{\tiny D}}aa^{\scriptsize\textcircled{\tiny D}}=a^{\scriptsize\textcircled{\tiny D}},$
then~$((a^{\scriptsize\textcircled{\tiny D}})^{\scriptsize\textcircled{\tiny D}})^{\scriptsize\textcircled{\tiny D}}=a^{\scriptsize\textcircled{\tiny D}}.$
\end{proof}
\vspace{4mm}

Let $a\in R^D$ with $i(a)=m$. The sum $a=c_a+n_a$ is called the core nilpotent decomposition \cite{D2004} of $a$, where $c_a=aa^Da$ is the core part of $a$, and~$n_a=(1-aa^D)a$ is the nilpotent part of $a$. This decomposition brings $n^m_a=0$, $c_an_a=n_ac_a=0$, and $c_a\in R^{\#}$ with $c_a^{\#}=a^D$ \cite{D2004}.
\begin{thm}~Let $a\in R$. Then $a\in R^{\scriptsize\textcircled{\tiny D}}$ with $I(a)=m$ if and only if $a\in R^D$ with $i(a)=m$ and $c_a\in R^{\tiny\textcircled{\tiny \#}}$. In this case, $a^{\scriptsize\textcircled{\tiny D}}$ coincides with $c_a^{\tiny\textcircled{\tiny \#}}$.
\end{thm}
\begin{proof}~Supposing $a\in R^{\scriptsize\textcircled{\tiny D}}$ with $I(a)=m$, we have $a\in R^D$ with $i(a)=m$ by Theorem 2.3 and
\begin{equation*}
\begin{aligned}
&a^{\scriptsize\textcircled{\tiny D}}c_a^2=a^{\scriptsize\textcircled{\tiny D}}(aa^Da)^2
  =a^{\scriptsize\textcircled{\tiny D}}a^3a^D=a^{\scriptsize\textcircled{\tiny D}}a^{m+2}(a^D)^m
  =a^{m+1}(a^D)^m=a^2a^D=aa^Da=c_a;\\
&c_a(a^{\scriptsize\textcircled{\tiny D}})^2=aa^Da(a^{\scriptsize\textcircled{\tiny D}})^2
  =aa^Da^{\scriptsize\textcircled{\tiny D}}=a^{\scriptsize\textcircled{\tiny D}};\\
&c_aa^{\scriptsize\textcircled{\tiny D}}=aa^Daa^{\scriptsize\textcircled{\tiny D}}
  =aa^{\scriptsize\textcircled{\tiny D}},~\mbox{which~implies~}(c_aa^{\scriptsize\textcircled{\tiny D}})^*=c_aa^{\scriptsize\textcircled{\tiny D}}.\\
\mbox{We~thus}&\mbox{~get}~c_a^{\tiny\textcircled{\tiny \#}}=a^{\scriptsize\textcircled{\tiny D}}.~~~~~~~~~~~~~~~~~~~~~~~~~~~~~~~~~~~~~~~~~~~~~~~~~~~~~~~~~~
\end{aligned}
\end{equation*}

Conversely, suppose $a\in R^D$ with $i(a)=m$, and suppose $c_a\in R^{\tiny\textcircled{\tiny \#}}$ which gives $c_a\in R^{\{1,3\}}$ and~$c_a^m\in R^{\{1,3\}}$.
Since $a=c_a+n_a,~c_an_a=n_ac_a=0$, and $n^m_a=0$, then $a^m=c^m_a$.
So we get $a^m\in R^{\{1,3\}}$, and $a^m(a^m)^{(1,3)}=c_a^m(c_a^m)^{(1,3)}=c_ac_a^{(1,3)}$. From Theorem 2.3, it follows that $a\in R^{\scriptsize\textcircled{\tiny D}}$, $a^{\scriptsize\textcircled{\tiny D}}=a^Da^m(a^m)^{(1,3)}=c_a^{\#}c_ac_a^{(1,3)}=c_a^{\tiny\textcircled{\tiny \#}}$ and $I(a)=m$.
\end{proof}
\vspace{4mm}

The notion of core-EP inverse \cite{DDDD2014} was introduced by Manjunatha Prasad and Mohana for a complex matrix which is not essentially of index 1 in 2014. This extends the notion of core inverse, which was initially defined for the matrices of index 1. Let $A, G\in \mathbb{C}^{n\times n}$, $G$ is a core-EP inverse of $A$ if $GAG=G$ and
$$\mathcal{C}(G)=\mathcal{C}(G^*)=\mathcal{C}(A^d),$$ where $d$ is the index of $A$ and $\mathcal{C}(G)$ is the column space of $G$.
The left annihilator of $a\in R$ is denoted by $^\circ{a}$ and is defined by $^\circ{a}=\{x\in R: xa=0\}$. We will also use the notation $aR=\{ax|x\in R\}$.

If we particularize the following Theorem 2.10 to the ring composed of complex $n\times n$ matrices, then $(1)\Leftrightarrow(2)$ indicates that the pseudo core inverse of a complex matrix coincides with its core-EP inverse. In other words, the notion of pseudo core inverse generalizes the notion of core-EP inverse from matrices to an arbitrary $*$-ring, in terms of equations.
\begin{thm}~Let $a,~x\in R$. Then the following conditions are equivalent:\\
$(1)$~$a^{\scriptsize\textcircled{\tiny D}}=x$;\\
$(2)$~$xax=x$ and $xR=x^*R=a^mR$ for some positive integer $m$;\\
$(3)$~$xax=x$, $xR=a^mR$ and $a^mR \subseteq x^*R$ for some positive integer $m$;\\
$(4)$~$xax=x$, $xR=a^mR$ and $^\circ(x^*)\subseteq\!\!~^\circ(a^m)$ for some positive integer $m$;\\
$(5)$~$xax=x$, $^\circ x=\!\!~^\circ(a^m)$ and $^\circ(x^*)\subseteq\!\!~^\circ(a^m)$ for some positive integer $m$.
\end{thm}

\begin{proof}~$(1)\Rightarrow(2)$ Suppose $a^{\scriptsize\textcircled{\tiny D}}=x$~with $I(a)=m$, then by the definition of pseudo core inverse, we have $xax=x$ and
\begin{equation*}
\begin{aligned}
&xR\subseteq~x^*R~\mbox{since}~x=ax^2=(ax)x=(ax)^*x=x^*a^*x\in x^*R;~~~~~~~~~~~~\\
&x^*R\subseteq a^mR~\mbox{since}~x^*=(xax)^*=axx^*=a^mx^mx^*\in a^mR;\\
&a^mR\subseteq xR~\mbox{since}~a^m=xa^{m+1}\in xR.
\end{aligned}
\end{equation*}
Thus, $xR=x^*R=a^mR$.

\noindent It ia easy to check that $(2)\Rightarrow(3)$, $(3)\Rightarrow(4)$ and $(4)\Rightarrow(5)$ hold.

\noindent $(5)\Rightarrow(1)$~Note that $xa-1\in\!\!~^\circ(x)$ and $^\circ(x)=\!\!~^\circ(a^m)$, then we have $xa^{m+1}=a^m$.

From~$x^*a^*-1\in\!\!~^\circ(x^*)$ and $^\circ(x^*)\subseteq\!\!~^\circ(a^m)$, it follows that $(x^*a^*-1)a^m=0$, i.e., $(ax)^*a^m=a^m$.
Post-multiply this equality by $a$, then we get
$(ax)^*aa^m=aa^m$, which implies $(ax)^*a-a\in\!\!~^\circ(a^m)=\!\!~^\circ(x)$. Thus $(ax)^*ax=ax$.
Therefore $(ax)^*=ax$.

The equalities $(ax)^*a^m=a^m$, $(ax)^*=ax$ and $^\circ(x)=\!\!~^\circ(a^m)$ yield that $ax^2=x$.

\noindent Hence, we get $a^{\scriptsize\textcircled{\tiny D}}=x$.
\end{proof}
\vspace{4mm}

Let us recall that an element $a\in R$ is called strongly $\pi$-regular in $R$ if there exist $x,~y\in R$ and positive integers $p,~q$ such that $$a^p=a^{p+1}x,~~~~a^q=ya^{q+1}.$$
\begin{lem}\emph{\cite{D1958}}~Let $a\in R$. Then $a\in R^D$ if and only if it is strongly $\pi$-regular in $R$.
\end{lem}

\begin{thm}~Let $a\in R$. Then $a\in R^{\scriptsize\textcircled{\tiny D}}$ if and only if
there exist $u,~v\in R$ and positive integers $p,~q$ such that $a^p=u(a^*)^{p+1}a^p,~a^q=va^{q+1}$.
\end{thm}
\begin{proof}~Suppose $a^{\scriptsize\textcircled{\tiny D}}=x$ with $I(a)=m$, then
$$xa^{m+1}=a^m,~~ ax^2=x,~~ (ax)^*=ax.$$
By Lemma 2.1, we have $a^mx^ma^m=a^m,~(a^mx^m)^*=a^mx^m$ which yields
$a^m=(a^mx^m)^*a^m.$ \\
Therefore~$a^m=(x^m)^*(a^m)^*a^m=(ax^{m+1})^*(a^m)^*a^m=(x^{m+1})^*(a^{*})^{m+1}a^m.$
Consequently,  the necessity holds.
\vspace{2mm}\\
\indent Conversely, suppose $u,~v\in R$ and positive integers $p,~q$ exist such that $a^p=u(a^*)^{p+1}a^p$, $a^q=va^{q+1}$. Since
$a^p=u(a^*)^{p+1}a^p\in R(a^p)^*a^p,$ then $au^*$ is a \{1,3\}-inverse of $a^p$.
So $a^p=a^pau^*a^p=a^{p+1}u^*a^p$, together with $a^q=va^{q+1}$, by Lemma 2.11, implies that $a\in R^D$ with $i(a)\leq p$.
By Theorem 2.3, we get $a\in R^{\scriptsize\textcircled{\tiny D}}$, moreover $a^{\scriptsize\textcircled{\tiny D}}=a^Da^p(a^p)^{(1,3)}
=a^Da^pau^*=a^pu^*$.
\end{proof}
\vspace{4mm}

From the proof of Theorem 2.12 and its dual case, we have the following result.
\begin{thm}~Let $a\in R$. Then the following conditions are equivalent:\\
$(1)$ $a\in R^D$ and $a^m\in R^{\dag}$ for some positive integer $m\geq i(a)$;\\
$(2)$ $a\in R^{\scriptsize\textcircled{\tiny D}}\cap R_{\scriptsize\textcircled{\tiny D}}$;\\
$(3)$ $a^m\in a^m(a^*)^{m+1}R\cap R(a^*)^{m+1}a^m$ for some positive integer $m$.
\end{thm}
\vspace{4mm}

Add one more equation $axa=a~(\text{resp.}~a^2x=a)$ to the three equations which exactly define the pseudo core inverse, then we can observe that $a\in R^{\tiny\textcircled{\tiny \#}}$ with $a^{\tiny\textcircled{\tiny \#}}=a^{\scriptsize\textcircled{\tiny D}}$.
\begin{prop}~Let $a,~x\in R$. Then $(1)\Leftrightarrow(2)$, and $(3)\Rightarrow(1)$, where \\
$(1)~a^{\tiny\textcircled{\tiny \#}}=x;$\\
$(2)~xa^{m+1}=a^m,~ax^2=x,~(ax)^{*}=ax,~axa=a$~for some positive integer $m$;\\
$(3)~xa^{m+1}=a^m,~ax^2=x,~(ax)^{*}=ax,~a^2x=a$~for some positive integer $m$.
\end{prop}

\begin{proof}~$(1)\Rightarrow(2)$~It is clear.
\vspace{2mm}\\
$(2)\Rightarrow(1)$~By~Lemma~2.1, we have $x=xax=xa^mx^m=x(xa^{m+1})x^m=x^2a^{m+1}x^m$,
~$\mbox{then}\\
x=x^{m+1}a^{2m}x^m~\mbox{by~induction}$.
So~$x=(x^{m+1}a^{m})a^mx^m=a^D(a^mx^m)=a^Dax$.
Thus,
$$a=axa=a(a^Dax)a=aa^D(axa)=aa^Da,$$
which~implies~$a\in R^{\#}.$
Hence~$xa^{m+1}=a^m$ becomes $xa^2=a$. From~$a^2x^2a=a=xa^2$, we get $a^{\#}=x^2a$.
Therefore~$a\in R^{\tiny\textcircled{\tiny \#}}~\text{with}~
a^{\tiny\textcircled{\tiny \#}}=a^{\#}aa^{(1,3)}=(x^2a)ax=xax=x$.
\vspace{2mm}\\
$(3)\Rightarrow(1)$~Since $a^2x=a$, then $a=a^{m+1}x^m$ by induction.
So we have $axa=a(xa^{m+1})x^m=a^{m+1}x^m=a.$
Therefore $(2)$ holds, then (1) holds.
\end{proof}

\begin{rem}~In Proposition $2.14$, $(1)$ may not imply $(3)$.

For example: take $R=\mathbb{C}^{2\times 2}$ with transpose as involution and let
$a=\begin{pmatrix}
           1&i\\
           0&0
   \end{pmatrix}\in R$. \\By a simple calculation,
$a^{\tiny\textcircled{\tiny \#}}=a^{\#}aa^{(1,3)}=\begin{pmatrix}
                                    1&0\\
                                    0&0
                                 \end{pmatrix}$,
but $a^2a^{\tiny\textcircled{\tiny \#}}=\begin{pmatrix}
                                                    1&0\\
                                                    0&0
                                                  \end{pmatrix}\neq a$.
\end{rem}
\vspace{0.4cm}

\section{Relations with other generalized inverses }
In this section, we wish to investigate the relationship between pseudo core inverse and other generalized inverses. Before that, let us mention two known definitions.
For $a, b\in R$, $a\leq_{\mathcal{H}}b \Leftrightarrow Ra\subseteq Rb ~\text{and}~aR\subseteq bR$ \cite{D2011}.
\begin{defn}\emph{\cite{D2011}}~Let $a, d, x\in R$. $x$ is the inverse of $a$ along $d$ if
$$xad=d=dax ~~~\text{and}~~~ x\leq_{\mathcal{H}}d.$$
\end{defn}

\begin{defn}\emph{\cite{D2012}}~Let $a, b, c, x\in R$. $x$ is the $(b,c)$-inverse of $a$ if
$$x\in bRx\cap xRc~~~\text{and}~~~ xab=b,~ cax=c.$$
\end{defn}

In what follows, we investigate the relationship between the pseudo core inverse and above introduced generalized inverses, respectively.
\begin{thm}~Let $a\in R$. $a$ is pseudo core invertible with $I(a)\leq m$ if and only if $a$ is invertible along $a^m(a^m)^{*}$, under the assumption $a^m\in R^{\{1,4\}}$. In this case, the pseudo core inverse of $a$ coincides with the inverse of $a$ along $a^m(a^m)^{*}$.
\end{thm}
\begin{proof}~Suppose $a^{\scriptsize\textcircled{\tiny D}}=x$ with $I(a)\leq m$, then~
$xa^{m+1}=a^m,~ax^2=x,~(ax)^{*}=ax.$\\
 Let $d=a^m(a^m)^{*}$, then we have
 \begin{equation*}
 \begin{aligned}
 xad&=xaa^m(a^m)^{*}=xa^{m+1}(a^m)^{*}=a^m(a^m)^{*}=d,\\ dax&=a^m(a^m)^{*}ax=a^m(a^m)^{*}(a^mx^m)^{*}=a^m(a^mx^ma^m)^{*}=a^m(a^m)^{*}=d.
\end{aligned}
\end{equation*}
Further,~$x = xax=x(ax)^{*}=x(a^mx^m)^{*}=x(x^m)^{*}(a^m)^{*}\\
\indent\indent\indent\!=x(x^m)^{*}(a^m(a^m)^{(1,4)}a^m)^{*}
=x(x^m)^{*}(a^m)^{(1,4)}a^m(a^m)^{*}\\
\indent\indent\indent\!\!=x(x^m)^{*}(a^m)^{(1,4)}d$, which implies $Rx\subseteq Rd$.\\

\indent\indent ~~$x=ax^2=a^mx^{m+1}=a^m(a^m)^{(1,4)}a^mx^{m+1}\\
\indent\indent\indent\!=a^m((a^m)^{(1,4)}a^m)^{*}x^{m+1}
=d((a^m)^{(1,4)})^{*}x^{m+1}$, which implies $xR\subseteq dR$.\\

\noindent From the above, $x$ is the inverse of $a$ along $ a^m(a^m)^*$.\\

Conversely,~suppose $x$ is the inverse of $a$ along $a^m(a^m)^*$, then we have $$(1)~xa^{m+1}(a^m)^{*}=a^m(a^m)^{*}=a^m(a^m)^{*}ax,
~(2)~x\in a^m(a^m)^{*}R \cap Ra^m(a^m)^{*}.$$
In order to show $a^{\scriptsize\textcircled{\tiny D}}=x$, we need to prove the following equalities.
\begin{equation*}
\begin{aligned}
xa^{m+1}&=xa^{m+1}(a^m)^{(1,4)}a^m=xa^{m+1}((a^m)^{(1,4)}a^m)^*=xa^{m+1}(a^m)^*((a^m)^{(1,4)})^*~~~~~~~~~~~~\\
&=a^m(a^m)^{*}((a^m)^{(1,4)})^*=a^m.
\end{aligned}
\end{equation*}

Since~$(a^m)^{*}=(a^m(a^m)^{(1,4)}a^m)^*
=(a^m)^{(1,4)}a^m(a^m)^*=(a^m)^{(1,4)}a^m(a^m)^*ax=(a^m)^*ax,$
then $a^m=(ax)^{*}a^m$.
There exists $u\in R$ such that $x=a^m(a^m)^{*}u$ since $x\in a^m(a^m)^{*}R$. So $ax=aa^m(a^m)^{*}u=a^ma(a^m)^{*}u=(ax)^{*}a^ma(a^m)^{*}u=(ax)^{*}aa^m(a^m)^{*}u=(ax)^{*}ax$.
Therefore $(ax)^{*}=ax$.
$$\text{From}~axa^m=a^m, \text{we~get~} ax^2=axa^m(a^m)^{*}u=a^m(a^m)^{*}u=x.~~~~~~~~~~~~~~~~~~~~~~~~~~~~~~$$

\noindent Hence, we get $a^{\scriptsize\textcircled{\tiny D}}=x$ with $I(a)\leq m$.
\end{proof}
\vspace{4mm}

\begin{thm}~Let $a\in R$.~$a$ is pseudo core invertible if and only if $a$ is $(a^m,(a^m)^{*})$ invertible for some positive integer $m$. In this case, the pseudo core inverse of $a$ coincides with the $(a^m,(a^m)^{*})$ inverse of $a$.
\end{thm}
\begin{proof}~Suppose $a^{\scriptsize\textcircled{\tiny D}}=x$ with $I(a)=m$, then we have~$$~xa^{m+1}=a^m,~~ax^2=x,~~(ax)^{*}=ax.$$
So~$x=ax^2=(ax)x=a^mx^mx\in a^mRx$,~$x=xax=x(a^mx^m)^*=x(x^m)^*(a^m)^*\in xR(a^m)^{*}$,\\
$xaa^m=xa^{m+1}=a^m$,~
$(a^m)^{*}ax=(a^m)^{*}$. Thus $x$ is the $(a^m,(a^m)^{*})$-inverse of $a$.\\

Conversely,~let $x$ be the $(a^m,(a^m)^{*})$ inverse of $a$, then there exist $s,~t\in R$ such that $$x=a^msx=xt(a^m)^{*},~~xaa^m=a^m,~~(a^m)^{*}ax=(a^m)^{*}.$$
So we have
\begin{equation*}
\begin{aligned}
&xa^{m+1}=a^m;\\
&(ax)^{*}=ax,~\mbox{due~to}~ax=aa^msx=a^{m}asx=(ax)^{*}a^masx=(ax)^{*}aa^msx=(ax)^{*}ax;\\
&ax^2=ax(a^msx)=(axa^m)sx=a^msx=x,~\mbox{on~account~of}~a^m=(ax)^{*}a^m=axa^m.
\end{aligned}
\end{equation*}
Thus, we get ~$a^{\scriptsize\textcircled{\tiny D}}=x$.
\end{proof}
\vspace{0.4cm}

\section{Reverse order law and additive property of the pseudo core inverse}
In this section, we show that the reverse order law for pseudo core inverse holds under certain conditions and investigate the pseudo core invertibility of the sum of two pseudo core invertible elements.

First, we give a crucial lemma.
\begin{lem}\emph{\cite{D2013}}~Let $a_i, b_i, c_i, y_i\in R~(i = 1, 2)$, and suppose that
each $a_i$ is $(b_i , c_i )$-invertible with $(b_i , c_i )$-inverse $y_i~(i = 1, 2)$. Then, for arbitrary
$d \in R$, $da_1 = a_2d$ and
$db_1 = b_2d$, $dc_1 = c_2d $
together imply that $y_{2}d = dy_{1}$.
\end{lem}

Applying the above lemma, we obtain the following result.
\begin{prop}~Let $a, x\in R$ with $ax=xa,~a^*x=xa^*$. If $a\in R^{\scriptsize\textcircled{\tiny D}}$, then $a^{\scriptsize\textcircled{\tiny D}}x=xa^{\scriptsize\textcircled{\tiny D}}$.
\end{prop}

\begin{proof}~We know that $a\in R^{\scriptsize\textcircled{\tiny D}}$ if and only if
$a$ is $(a^m, (a^m)^*)$ invertible for some positive integer $m$ and in that case, the pseudo core inverse of $a$ coincides with the $(a^m,(a^m)^{*})$ inverse of $a$ (see the above Theorem 3.4).
From the condition $ax=xa,~a^*x=xa^*$, we have $a^mx=xa^m,~(a^m)^*x=x(a^m)^*$. According to Lemma~4.1, we get $a^{\scriptsize\textcircled{\tiny D}}x=xa^{\scriptsize\textcircled{\tiny D}}$.
\end{proof}
\vspace{4mm}

Applying Proposition 4.2, we obtain the following theorem.
\begin{thm}~Let $a, b \in R^{\scriptsize\textcircled{\tiny D}}$ with $ab=ba$ and $ab^*=b^*a$. Then $(ab)^{\scriptsize\textcircled{\tiny D}}=a^{\scriptsize\textcircled{\tiny D}}b^{\scriptsize\textcircled{\tiny D}}=b^{\scriptsize\textcircled{\tiny D}}a^{\scriptsize\textcircled{\tiny D}}$.
\end{thm}

\begin{proof}~From Proposition 4.2, it follows that $$b^{\scriptsize\textcircled{\tiny D}}a=ab^{\scriptsize\textcircled{\tiny D}}~~~\text{and}~~~a^{\scriptsize\textcircled{\tiny D}}b=ba^{\scriptsize\textcircled{\tiny D}}.$$
The condition $b^{*}a=ab^{*}, ~a^{*}b^{*}=b^{*}a^{*}$ ensures that $b^{*}a^{\scriptsize\textcircled{\tiny D}}=a^{\scriptsize\textcircled{\tiny D}}b^{*}$, which together with $a^{\scriptsize\textcircled{\tiny D}}b=ba^{\scriptsize\textcircled{\tiny D}}$, implies that $a^{\scriptsize\textcircled{\tiny D}}b^{\scriptsize\textcircled{\tiny D}}=b^{\scriptsize\textcircled{\tiny D}}a^{\scriptsize\textcircled{\tiny D}}$.

Let~$t=$max$\{I(a), I(b)\}$, then we have
\begin{equation*}
\begin{aligned}
&b^{\scriptsize\textcircled{\tiny D}}a^{\scriptsize\textcircled{\tiny D}}
(ab)^{t+1}=b^{\scriptsize\textcircled{\tiny D}}a^{\scriptsize\textcircled{\tiny D}}a^{t+1}b^{t+1}
=b^{\scriptsize\textcircled{\tiny D}}a^tb^{t+1}=a^tb^{\scriptsize\textcircled{\tiny D}}b^{t+1}=a^tb^t=(ab)^t;\\
&ab(b^{\scriptsize\textcircled{\tiny D}}a^{\scriptsize\textcircled{\tiny D}})^2=ab(b^{\scriptsize\textcircled{\tiny D}})^2(a^{\scriptsize\textcircled{\tiny D}})^2
=ab^{\scriptsize\textcircled{\tiny D}}(a^{\scriptsize\textcircled{\tiny D}})^2 =b^{\scriptsize\textcircled{\tiny D}}a(a^{\scriptsize\textcircled{\tiny D}})^2 =b^{\scriptsize\textcircled{\tiny D}}a^{\scriptsize\textcircled{\tiny D}};~~~~~\\
&(abb^{\scriptsize\textcircled{\tiny D}}a^{\scriptsize\textcircled{\tiny D}})^{*}
=(aa^{\scriptsize\textcircled{\tiny D}}bb^{\scriptsize\textcircled{\tiny D}})^{*}
=(bb^{\scriptsize\textcircled{\tiny D}})^{*}(aa^{\scriptsize\textcircled{\tiny D}})^{*}
=bb^{\scriptsize\textcircled{\tiny D}}aa^{\scriptsize\textcircled{\tiny D}}
=abb^{\scriptsize\textcircled{\tiny D}}a^{\scriptsize\textcircled{\tiny D}}.
\end{aligned}
\end{equation*}
Thus $(ab)^{\scriptsize\textcircled{\tiny D}}
=b^{\scriptsize\textcircled{\tiny D}}a^{\scriptsize\textcircled{\tiny D}}
=a^{\scriptsize\textcircled{\tiny D}}b^{\scriptsize\textcircled{\tiny D}}$.
\end{proof}
\vspace{4mm}

Next, we explore the pseudo core invertibility of the sum of two pseudo core invertible elements.
\begin{thm}~Let $a, b\in R^{\scriptsize\textcircled{\tiny D}}$ with $ab=ba=0,~ a^{*}b=0$. Then $a+b\in R^{\scriptsize\textcircled{\tiny D}}$ with $(a+b)^{\scriptsize\textcircled{\tiny D}}=a^{\scriptsize\textcircled{\tiny D}}+b^{\scriptsize\textcircled{\tiny D}}$.
\end{thm}
\begin{proof}~Since $a, b\in R^{\scriptsize\textcircled{\tiny D}}$, by Lemma 2.1, they are Drazin invertible, and
$(a+b)^D=a^D+b^D$ under the condition $ab=ba=0$.

Again by the hypothesis $ab=ba=0,~ a^{*}b=0$, we find
\begin{equation*}
\begin{aligned}
&ab^{\scriptsize\textcircled{\tiny D}}=ab(b^{\scriptsize\textcircled{\tiny D}})^2=0,~ba^{\scriptsize\textcircled{\tiny D}}=ba(a^{\scriptsize\textcircled{\tiny D}})^2=0,\\
&b^{\scriptsize\textcircled{\tiny D}}a=b^{\scriptsize\textcircled{\tiny D}}(b^{\scriptsize\textcircled{\tiny D}})^{*}b^{*}a=0,~
a^{\scriptsize\textcircled{\tiny D}}b=a^{\scriptsize\textcircled{\tiny D}}(a^{\scriptsize\textcircled{\tiny D}})^{*}a^{*}b=0,\\
&a^{\scriptsize\textcircled{\tiny D}}b^{\scriptsize\textcircled{\tiny D}}
=a^{\scriptsize\textcircled{\tiny D}}(a^{\scriptsize\textcircled{\tiny D}})^{*}a^{*}b(b^{\scriptsize\textcircled{\tiny D}})^{2}=0,~
b^{\scriptsize\textcircled{\tiny D}}a^{\scriptsize\textcircled{\tiny D}}
=b^{\scriptsize\textcircled{\tiny D}}(b^{\scriptsize\textcircled{\tiny D}})^{*}b^{*}a(a^{\scriptsize\textcircled{\tiny D}})^{2}=0.
\end{aligned}
\end{equation*}

Let $m=$max$\{I(a), I(b)\}$, then $a^m(a^{\scriptsize\textcircled{\tiny D}})^ma^m=a^m$ and
$b^m(b^{\scriptsize\textcircled{\tiny D}})^mb^m=b^m$, \\so
\begin{equation*}
\begin{aligned}
(a+b)^m((a^{\scriptsize\textcircled{\tiny D}})^m+(b^{\scriptsize\textcircled{\tiny D}})^m)
&=(a^m+b^m)((a^{\scriptsize\textcircled{\tiny D}})^m+(b^{\scriptsize\textcircled{\tiny D}})^m)
~~~~~~~~~~~~~~~~~~~~~~~~~~\\
&=a^m(a^{\scriptsize\textcircled{\tiny D}})^m
+b^m(b^{\scriptsize\textcircled{\tiny D}})^m\\
&=aa^{\scriptsize\textcircled{\tiny D}}+bb^{\scriptsize\textcircled{\tiny D}}.
\end{aligned}
\end{equation*}
$\mbox{Then,}~((a+b)^m((a^{\scriptsize\textcircled{\tiny D}})^m+(b^{\scriptsize\textcircled{\tiny D}})^m))^{*}
=(a+b)^m((a^{\scriptsize\textcircled{\tiny D}})^m+(b^{\scriptsize\textcircled{\tiny D}})^m).~~~~~~
$
\begin{equation*}
\begin{aligned}
\mbox{Further,~}(a+b)^m((a^{\scriptsize\textcircled{\tiny D}})^m+(b^{\scriptsize\textcircled{\tiny D}})^m)(a+b)^m
&=(a^m(a^{\scriptsize\textcircled{\tiny D}})^m+b^m(b^{\scriptsize\textcircled{\tiny D}})^m)(a^m+b^m)~~~~~~~~~~~~~~~~~~~~~~~~~~~~~~~~~~~~~~~~~~~~~~~~~~~~~~\\
&=a^m(a^{\scriptsize\textcircled{\tiny D}})^ma^m+b^m(b^{\scriptsize\textcircled{\tiny D}})^mb^m\\
&=a^m+b^m=(a+b)^m.
\end{aligned}
\end{equation*}
Hence $(a^{\scriptsize\textcircled{\tiny D}})^m+(b^{\scriptsize\textcircled{\tiny D}})^m$ is a \{1,3\}-inverse of $(a+b)^m$.\\
\begin{equation*}
\begin{aligned}
\mbox{Therefore,~we~get}~(a+b)^{\scriptsize\textcircled{\tiny D}}
=&~(a+b)^D(a+b)^m((a+b)^m)^{(1,3)}~~~~~~~~~~~~~~~~~~~~~~~~~~~~~~~~~~~~~~~~~~~\\
=&~(a^D+b^D)(a^m+b^m)((a^{\scriptsize\textcircled{\tiny D}})^m+(b^{\scriptsize\textcircled{\tiny D}})^m)~~~~~~~~~~~~~~~~~~~~~~~\\
=&~a^Da^m(a^m)^{(1,3)}+b^Db^m(b^m)^{(1,3)}\\
=&a^{\scriptsize\textcircled{\tiny D}}+b^{\scriptsize\textcircled{\tiny D}}.
\end{aligned}
\end{equation*}
\end{proof}

\begin{rem}
 It is noteworthy that condition $ab=0,~a^{*}b=0$ (without $ba=0$) is not sufficient to show the pseudo core invertibility of $a+b$ although both $a$ and $b$ are pseudo core invertible.

For example:~by setting $R=\mathbb{C}^{2\times 2}$ with transpose as its involution,\\
$$a=\begin{pmatrix}
i & 0\\
0 & 0
\end{pmatrix},~~~~b=\begin{pmatrix}
                  0 & 0\\
                  -1 & 0
                  \end{pmatrix},$$
we have~ $ab=a^{*}b=0,~\text{but}~ba\neq0$.

Observe that $a^{\#}=-a$ and~$aa^{(1,3)}=\begin{pmatrix}
                        1 & 0\\
                        0 & 0
                        \end{pmatrix}$,~which imply~$$a^{\scriptsize\textcircled{\tiny D}}=a^{\tiny\textcircled{\tiny \#}}=a^{\#}aa^{(1,3)}
                                     =\begin{pmatrix}
-i & 0\\
0 & 0
\end{pmatrix}.$$

It is obvious that $b^{\scriptsize\textcircled{\tiny D}}=0$.

As for $a+b=\begin{pmatrix}
     i & 0\\
    -1 & 0
    \end{pmatrix}$, by calculation, we find that neither $a+b$ nor $(a+b)^2$ has any \{1,3\}-inverse. Since
$(a+b)^m=\begin{cases}
         (-1)^{\frac{m-1}{2}}(a+b)& m ~\text{is odd} \\
         (-1)^{\frac{m}{2}+1}(a+b)^2 & m ~\text{is even}
\end{cases}$, we conclude that
$(a+b)^m$ has no \{1,3\}-inverse for arbitrary positive integer $m$.

Hence $a+b$ is not pseudo core invertible.
\end{rem}
\vspace{4mm}

\section{Computations for the pseudo core inverses of complex matrices}
Lastly, one may take an interest in how to compute the pseudo core inverse of a square complex matrix. Here we exhibit two methods.

For any matrix $A \in \mathbb{C}^{n \times n}$ of rank $r>0$ the Hartwig-Spindelb\"{o}ck decomposition \cite{D1984} is given by
\begin{equation}
\begin{aligned}
A=U\begin{pmatrix}
     \Sigma K & \Sigma L\\
     0        &  0
     \end{pmatrix}U^{*},
\end{aligned}
\end{equation}
where $U \in \mathbb{C}^{n \times n}$ is unitary, $\Sigma=$diag$(\sigma_1I_{r_1}, \sigma_2I_{r_2}, \cdots, \sigma_tI_{r_t})$ is a diagonal matrix, the diagonal entries $\sigma_i$ being singular values of $A$, $\sigma_1 > \sigma_2 > \cdots > \sigma_t > 0,~ r_1+r_2+\cdots+r_t=r$ and $K\in \mathbb{C}^{r \times r}, ~L\in \mathbb{C}^{r \times {(n-r)}}$ satisfying $KK^{*}+LL^{*}=I_r$.

\begin{thm}~Let $A\in\mathbb{C}^{n \times n}$ be of the form $(1)$. Then $A^{\scriptsize\textcircled{\tiny D}}
=U\begin{pmatrix}
(\Sigma K)^{\scriptsize\textcircled{\tiny D}} & 0\\
0  & 0
      \end{pmatrix}U^{*}.$
\end{thm}
\begin{proof}
Suppose~$m$ is the index of $A$, then $I(A)=m$.
Suppose $A^{\scriptsize\textcircled{\tiny D}}=X$, then we have
$$XA^{m+1}=A^m,~~ AX^2=X,~~ (AX)^{*}=AX.$$
The equality $XA^{m+1}=A^m$ ensures that
there exists $X=\begin{pmatrix}
      X_1 & X_2\\
      X_3 & X_4
      \end{pmatrix} \in \mathbb{C}^{n\times n}$ such that
$$\begin{pmatrix}
      X_1 & X_2\\
      X_3 & X_4
      \end{pmatrix}\begin{pmatrix}
     (\Sigma K)^{m+1} & (\Sigma K)^{m}\Sigma L\\
     0        &  0
     \end{pmatrix}=\begin{pmatrix}
     (\Sigma K)^m & (\Sigma K)^{m-1}\Sigma L\\
     0        &  0
     \end{pmatrix}.$$
Then we have $$\begin{cases}(\Sigma K)^m &=X_1(\Sigma K)^{m+1},\\
(\Sigma K)^{m-1}\Sigma L &=X_1(\Sigma K)^{m}\Sigma L.
\end{cases}$$
$$\text{So}~(\Sigma K)^{m-1}\Sigma(K, L)
=X_1(\Sigma K)^m\Sigma(K, L).~~~~~~~~~~~~~~~~~~~~~~~~~~~~~~~~~~~~~~~~~~~~~~~~~~~~~~~~~~~~~~~~~~~~~~$$
Right multiply the above equality by $\begin{pmatrix}
                                      K^{*}\\
                                      L^{*}
                                     \end{pmatrix}$, then
$$(\Sigma K)^{m-1}\Sigma \begin{pmatrix}
                               K & L
                          \end{pmatrix}\begin{pmatrix}
                                      K^{*}\\
                                      L^{*}
                                     \end{pmatrix}=X_1(\Sigma K)^m\Sigma\begin{pmatrix}
                                                                   K & L
                                                                  \end{pmatrix}\begin{pmatrix}
                                                                                K^{*}\\
                                                                                L^{*}
                                                                               \end{pmatrix}.$$
Since $KK^{*}+LL^{*}=I$, we obtain $(\Sigma K)^{m-1}\Sigma=X_1(\Sigma K)^m\Sigma.$\\
Then~$(\Sigma K)^{m-1}=X_1(\Sigma K)^m.$\\
Thus~rank $((\Sigma K)^{m-1})=\text{rank~}((\Sigma K)^{m})$, which implies that
$I(\Sigma K)\leq m-1$.\\
Therefore
\begin{equation}
(\Sigma K)^{\scriptsize\textcircled{\tiny D}}=(\Sigma K)^D(\Sigma K)^{m-1}((\Sigma K)^{m-1})^{\dag}.
\end{equation}

The Drazin inverse of $A$ is given as \cite{D22014}
$$A^D=U\begin{pmatrix}
      (\Sigma K)^D & ((\Sigma K)^D)^2 \Sigma L\\
      0            & 0
      \end{pmatrix}U^{*}.$$
Since~$A^m=\left(U\begin{pmatrix}
     \Sigma K & \Sigma L\\
     0        &  0
     \end{pmatrix}U^{*}\right)^m=U\begin{pmatrix}
     \Sigma K & \Sigma L\\
     0        &  0
     \end{pmatrix}^mU^{*}=U\begin{pmatrix}
     (\Sigma K)^m & (\Sigma K)^{m-1}\Sigma L\\
     0        &  0
     \end{pmatrix}U^{*}$, we have~
$(A^m)^{\dag}=U\begin{pmatrix}
     (\Sigma K)^m & (\Sigma K)^{m-1}\Sigma L\\
     0        &  0
     \end{pmatrix}^{\dag}U^{*}.~~~~~~~~~~~~~~~~~~~\\
\mbox{Thus~~} A^m(A^m)^{\dag}=U\begin{pmatrix}
     (\Sigma K)^m & (\Sigma K)^{m-1}\Sigma L\\
     0        &  0
     \end{pmatrix}\begin{pmatrix}
     (\Sigma K)^m & (\Sigma K)^{m-1}\Sigma L\\
     0        &  0
     \end{pmatrix}^{\dag}U^{*}\\
\indent\quad\indent\indent\indent\!=U\begin{pmatrix}
     B & C\\
     0 & 0
     \end{pmatrix}\begin{pmatrix}
     B & C\\
     0 & 0
     \end{pmatrix}^{\dag}U^{*},~
\mbox{where~} B=(\Sigma K)^m,~C=(\Sigma K)^{m-1}\Sigma L.$\\
From \cite{D1975}, we know $\begin{pmatrix}
     0 & 0\\
     B & C
     \end{pmatrix}^{\dag}=\begin{pmatrix}
     0 & B^{*}L^{\dag}\\
     0 & C^{*}L^{\dag}
     \end{pmatrix}$, where $L=BB^{*}+CC^{*}$.

$$\mbox{So}~\begin{pmatrix}
     B & C\\
     0 & 0
     \end{pmatrix}^{\dag}=\left[\begin{pmatrix}
     0 & 1\\
     1 & 0
     \end{pmatrix}\begin{pmatrix}
     0 & 0\\
     B & C
     \end{pmatrix}\right]^{\dag}=\begin{pmatrix}
     0 & 0\\
     B & C
     \end{pmatrix}^{\dag}\begin{pmatrix}
     0 & 1\\
     1 & 0
     \end{pmatrix}=\begin{pmatrix}
    B^{*}L^{\dag} &0\\
    C^{*}L^{\dag} &0
     \end{pmatrix}.~~~~~~~~~~~~~~~~~~~~~~~~~~~~~~~$$
Then
\begin{equation*}
\begin{aligned}
\begin{pmatrix}
     B & C\\
     0 & 0
     \end{pmatrix}\begin{pmatrix}
     B & C\\
     0 & 0
     \end{pmatrix}^{\dag}&=\begin{pmatrix}
     B & C\\
     0 & 0
     \end{pmatrix}\begin{pmatrix}
    B^{*}L^{\dag} &0\\
    C^{*}L^{\dag} &0
     \end{pmatrix}=\begin{pmatrix}
    BB^{*}L^{\dag}+CC^{*}L^{\dag} &0\\
     0                            &0
     \end{pmatrix}\\
&=\begin{pmatrix}
LL^{\dag} &0\\
0         &0
\end{pmatrix},~\mbox{where}~L=BB^{*}+CC^{*}.
\end{aligned}
\end{equation*}
Therefore $A^m(A^m)^{\dag}=U\begin{pmatrix}
LL^{\dag} &0\\
0         &0
\end{pmatrix}U^{*}$,~where
\begin{equation*}
\begin{aligned}
L&=BB^{*}+CC^{*}\\
&=(\Sigma K)^m((\Sigma K)^m)^{*}+(\Sigma K)^{m-1}\Sigma L((\Sigma K)^{m-1}\Sigma L)^{*}~~~~~~~~~~~~~~~~~~~\\
&=(\Sigma K)^{m-1}[\Sigma K(\Sigma K)^{*}+\Sigma L(\Sigma L)^{*}]((\Sigma K)^{m-1})^{*}\\
&=(\Sigma K)^{m-1}\Sigma\Sigma^{*}((\Sigma K)^{m-1})^{*}~~~~~~~~~~~~~~~~\\
&=(\Sigma K)^{m-1}\Sigma((\Sigma K)^{m-1}\Sigma)^{*}\\
&=TT^{*},~ \mbox{where~} T=(\Sigma K)^{m-1}\Sigma.
\end{aligned}
\end{equation*}
Then $LL^{\dag}=TT^{*}(TT^{*})^{\dag}=TT^{\dag}
=(\Sigma K)^{m-1}\Sigma((\Sigma K)^{m-1}\Sigma)^{\dag}\\
\indent\indent\quad~=T_1\Sigma(T_1\Sigma)^{\dag}$,~where~$T_1=(\Sigma K)^{m-1}.$\\
By \cite{D1991}, we get
$$(T_1\Sigma)^{\dag}=(T_1\Sigma)^{*}[(T_1\Sigma(T_1\Sigma)^{*}+I-T_1T_1^{\dag}]^{-1}.$$
So $T_1\Sigma(T_1\Sigma)^{\dag}
=T_1\Sigma(T_1\Sigma)^{*}[(T_1\Sigma(T_1\Sigma)^{*}+I-T_1T_1^{\dag}]^{-1}\\
\indent\indent\indent~~~~ =T_1T_1^{\dag}[T_1\Sigma(T_1\Sigma)^{*}+I-T_1T_1^{\dag}][(T_1\Sigma(T_1\Sigma)^{*}+I-T_1T_1^{\dag}]^{-1}
 =T_1T_1^{\dag}$.\\
Therefore $A^m(A^m)^{\dag}=U\begin{pmatrix}
T_1T_1^{\dag} &0\\
0         &0
\end{pmatrix}U^{*}$, where $T_1=(\Sigma K)^{m-1}$.\\
\mbox{Hence~}
\begin{equation*}
\begin{aligned}
A^{\scriptsize\textcircled{\tiny D}}
&=A^DA^m(A^m)^{\dag}
=U\begin{pmatrix}
      (\Sigma K)^D & ((\Sigma K)^D)^2 \Sigma L\\
      0            & 0
      \end{pmatrix}\begin{pmatrix}
T_1T_1^{\dag} &0\\
0         &0
\end{pmatrix}U^{*}~~~~~~~~~~~~~~~~~~~~~~~~~~~\\
&=U\begin{pmatrix}
      (\Sigma K)^DT_1T_1^{\dag} & 0\\
      0                     & 0
      \end{pmatrix}U^{*}
=U\begin{pmatrix}
      (\Sigma K)^D(\Sigma K)^{m-1}((\Sigma K)^{m-1})^{\dag} & 0\\
      0                     & 0
      \end{pmatrix}U^{*}\\
&\xlongequal{(2)}U\begin{pmatrix}
(\Sigma K)^{\scriptsize\textcircled{\tiny D}} & 0\\
0  & 0
      \end{pmatrix}U^{*}.
\end{aligned}
\end{equation*}
\end{proof}

Let $A\in \mathbb{C}^{n\times n}$ of rank $r>0$, then there exists invertible~$P\in \mathbb{C}^{n\times n}$ such that
\begin{equation}
A=P^{-1}\begin{pmatrix}
          D &O\\
          O &N
         \end{pmatrix}P,
\end{equation}
$\mbox{where}~D\in \mathbb{C}^{r\times r}~\mbox{is~invertible}$ and$~N\in \mathbb{C}^{(n-r)\times (n-r)}\mbox{~is~nilpotent}.$~
Suppose $P=\begin{pmatrix}
                                    P_1 \\
                                    P_2
                                    \end{pmatrix}$ and $P^{-1}=(Q_1, Q_2)$.
\begin{thm}~Let $A\in\mathbb{C}^{n \times n}$ be of the form $(3)$. Then $A^{\scriptsize\textcircled{\tiny D}}=Q_1D^{-1}({Q_1}^{*} Q_1)^{-1}{Q_1}^{*}$.
\end{thm}
\begin{proof}~
From $P=\begin{pmatrix}
                                    P_1 \\
                                    P_2
                                    \end{pmatrix}$,~$P^{-1}=(Q_1, Q_2)$ and $A=P^{-1}\begin{pmatrix}
          D &O\\
          O &N
         \end{pmatrix}P$, it follows that
$$A^D=P^{-1}\begin{pmatrix}
                                    D^{-1} &O\\
                                    O      &O
                                     \end{pmatrix}P=Q_1D^{-1}P_1~\text{and}~P_1Q_1=I.$$
Let $m$ be the positive integer such that $N^m=0$, then
$$A^m=(Q_1, Q_2)\begin{pmatrix}
                                    D^{m} &O\\
                                    O      &O
                                     \end{pmatrix}\begin{pmatrix}
                                                     P_1 \\
                                                     P_2
                                                   \end{pmatrix}=Q_1D^mP_1.$$
Setting $B=Q_1(D^{-1})^m({Q_1}^{*} Q_1)^{-1}{Q_1}^{*}$, we have
 $$A^mB=Q_1D^mP_1Q_1(D^{-1})^m({Q_1}^{*} Q_1)^{-1}{Q_1}^{*}
=Q_1({Q_1}^{*} Q_1)^{-1}{Q_1}^{*}.$$
So $(A^mB)^{*}=A^mB$ and
$A^mBA^m=Q_1({Q_1}^{*} Q_1)^{-1}{Q_1}^{*}Q_1D^mP_1=Q_1D^mP_1=A^m$.
Namely, $B$ is a \{1,3\}-inverse of $A^m$.\\
\mbox{Hence~}
\begin{equation*}
\begin{aligned}
A^{\scriptsize\textcircled{\tiny D}}&=A^DA^m(A^m)^{(1,3)}=Q_1D^{-1}P_1Q_1({Q_1}^{*} Q_1)^{-1}{Q_1}^{*}\\
&=Q_1D^{-1}({Q_1}^{*} Q_1)^{-1}{Q_1}^{*}.
\end{aligned}
\end{equation*}
\end{proof}
\vspace{0.8cm}

\noindent {\large\bf Acknowledgements}\\
This research is supported by the National Natural Science Foundation
of China (No.11371089), the Scientific Innovation Research of College Graduates in Jiangsu Province (No.KYZZ16$\_$0112),
the Natural Science Foundation of Jiangsu Province (No.BK20141327).

\end{document}